\documentclass{article}

\usepackage[preprint,nonatbib]{Preprint_ICUS}
\usepackage{amssymb}
\usepackage{amsthm}
\usepackage{mathtools}
\usepackage{amsfonts}
\usepackage{amsmath}

\usepackage[utf8]{inputenc} % allow utf-8 input
\usepackage[T1]{fontenc}    % use 8-bit T1 fonts
\usepackage{hyperref}       % hyperlinks
\usepackage{url}            % simple URL typesetting
\usepackage{booktabs}       % professional-quality tables
\usepackage{amsfonts}       % blackboard math symbols
\usepackage{nicefrac}       % compact symbols for 1/2, etc.
\usepackage{microtype}      % microtypography
\usepackage{xcolor}         % colors

\newtheorem{theorem}{Theorem}[section]
\newtheorem{definition}[theorem]{Definition}

\newtheorem{lemma}[theorem]{Lemma}

\newtheorem{corollary}[theorem]{Corollary}

\title{Exponential finite sample bounds for incomplete U-statistics}

\author{%
	Andreas Maurer \\
	Istituto Italiano di Tecnologia, 16163 Genoa, Italy\\
	\texttt{am@andreas-maurer.eu}
}

\begin{document}

	\maketitle

	\begin{abstract}
Incomplete U-statistics have been proposed to accelerate computation. They use only a subset of the subsamples required for kernel evaluations by complete U-statistics. This paper gives a finite sample bound in the style of Bernstein's inequality. Applied to complete U-statistics the resulting inequality improves over the bounds of both Hoeffding and Arcones. For randomly determined subsamples it is shown, that as soon as the their number reaches the square of the sample-size, the same order bound is obtained as for the complete statistic.\end{abstract}
	
\section{Introduction}

Let $\mu $ be a probability measure on a measurable space $\mathcal{X}$, $m$
an integer and $K:\mathcal{X}^{m}\rightarrow \mathbb{R}$ a measurable,
symmetric, bounded kernel. We wish to estimate the parameter $\theta =\theta
\left( \mu \right) =\mathbb{E}\left[ K\left( X_{1},...,X_{m}\right) \right] $
from a finite sample $\mathbf{X}=\left( X_{1},...,X_{n}\right) \sim \mu ^{n}$
, where the number $n$ of independent observations is much larger than the
degree $m$ of the kernel. The standard estimator is the U-statistic%
\begin{equation*}
U\left( \mathbf{X}\right) =\binom{n}{m}^{-1}\sum_{1\leq i_{1}<...<i_{m}\leq
	n}K\left( X_{i_{1}},...,X_{i_{m}}\right) .
\end{equation*}%
The dependence on $K$, which should be understood in most cases, will not be
made explicit. We also introduce the shorthand 
\begin{equation*}
U\left( \mathbf{X}\right) =\binom{n}{m}^{-1}\sum_{W\subset \left[ n\right]
	,\left\vert W\right\vert =m}K\left( \mathbf{X}^{W}\right) ,
\end{equation*}%
where $K\left( \mathbf{X}^{W}\right) =K\left(
X_{i_{1}},X_{i_{2}},...,X_{i_{m}}\right) $ where $\left(
i_{1},i_{2},...,i_{m}\right) $ is an enumeration of $W$, which is arbitrary
by the symmetry of $K$.

The U-statistic is an unbiased estimator of $\theta $, and it has minimal
variance among all unbiased estimators of $\theta $ (\cite{halmos1946theory}%
). On the other hand computational requirements are excessive for large $m$
and $n$. Since the number of subsets scales as $n^{m}$ for $n=10^{4}$ and $%
m=5$ the order of necessary kernel evaluations is already about $10^{18}$
and out of reach for present computing power.

Since the high degree of dependence between the terms suggests, that a
smaller number of kernel evaluations will already lead to an acceptable
estimation error with lowered computational burden, incomplete U-statistics
were proposed. Let $\mathbf{W}=\left( W_{1},...,W_{M}\right) \in \left\{
W\subset \left[ n\right] :\left\vert W\right\vert =m\right\} ^{M}$ be a
sequence of $M$ subsets of $\left\{ 1,...,n\right\} $ of cardinality $m$,
and define the incomplete U-statistic%
\begin{equation*}
U_{\mathbf{W}}\left( \mathbf{X}\right) =\frac{1}{M}\sum_{i=1}^{M}K\left( 
\mathbf{X}^{W_{i}}\right) .
\end{equation*}%
The sequence $\mathbf{W}$ is called the \textit{design} \cite{kong2020design}
and may be chosen at random, for example by sampling from $\left\{ W\subset %
\left[ n\right] :\left\vert W\right\vert =m\right\} $. The number $M$ of
required kernel evaluations can be interpreted as a \textit{computational
	budget}.

In this work we address the question how small $M$ can be and how $\mathbf{W}
$ has to be chosen to obtain with overwhelming probability an acceptable
bound on the estimation error for a finite sample of given size $n$. We will
give variance dependent exponential finite sample bounds, modeled after the
classical Bernstein inequality. If applied to complete U-statistics, which
are a special case of incomplete ones, these bounds improve over the
classical result of Hoeffding \cite{hoeffding58probability}\ and over the
more recent one of Arcones \cite{arcones1995bernstein}. If $\mathbf{W}%
=\left( W_{1},...,W_{M}\right) $ is sampled with replacement from the
uniform distribution on $\left\{ W\subset \left[ n\right] :\left\vert
W\right\vert =m\right\} $, we show that $M=n^{2}$ samples suffice to obtain
with high probability a bound of the same order.

We assume the kernel to be bounded throughout and set $K:\mathcal{X}%
^{m}\rightarrow \left[ -1,1\right] $ to simplify statements. Results for
different values of $\left\Vert K\right\Vert _{\infty }$ follow from
re-scaling. The next section introduces some notation and gives a brief
historical review of literature on the subject. Section \ref{Section Results}
states our results and Section \ref{Section Proofs} contains the
proofs.\bigskip 

\section{Preliminaries}

The symbol $\left\vert .\right\vert $ is used both for the absolute value of
real numbers and the cardinality of sets. We use lower-case letters for
scalars, upper-case letters for random variables and bold letters for
vectors. For $n\in \mathbb{N}$, $\left[ n\right] $ denotes the set $\left\{
1,...,n\right\} $. If $\mathcal{X}$ is a set, $y\in \mathcal{X}$, $k\in %
\left[ n\right] $, the substitution operator is defined by%
\begin{equation*}
S_{y}^{k}:\mathbf{x}\in \mathcal{X}^{n}\mapsto S_{y}^{k}\left( \mathbf{x}%
\right) =\left( x_{1},...,x_{k-1},y,x_{k+1},...,x_{n}\right) \in \mathcal{X}%
^{n}\text{.}
\end{equation*}%
We also write 
\begin{equation*}
\mathbf{x}^{\backslash k}=\left( x_{1},...,x_{k-1},x_{k+1},...,x_{n}\right)
\in \mathcal{X}^{n-1}.
\end{equation*}%
For $f:\mathcal{X}^{n}\rightarrow \mathbb{R}$, $y,y^{\prime }\in \mathcal{X}$
and $k\in \left[ n\right] $ we introduce the partial difference operator

\begin{equation*}
D_{y,y^{\prime }}^{k}f\left( \mathbf{x}\right) =f\left( S_{y}^{k}\mathbf{x}%
\right) -f\left( S_{y^{\prime }}^{k}\mathbf{x}\right) .
\end{equation*}%
Throughout this work $1\leq m<n$ will be fixed integers, and $K$ will be a
measurable kernel $K:\mathcal{X}^{m}\rightarrow \left[ -1,1\right] $. The
kernel $K$ will be assumed permutation symmetric unless otherwise stated. $%
\mathbf{X}=\left( X_{1},...,X_{n}\right) $ will be an iid vector distributed
in $\mathcal{X}^{n}$, that is $\mathbf{X}\sim \mu ^{n}$, where $\mu $ is a
probability measure on $\mathcal{X}$ and $\mathbf{X}^{\prime }$ will be an
independent copy of $\mathbf{X}$. For $K:\mathcal{X}^{m}\rightarrow \left[
-1,1\right] $ and $k\in \left[ m\right] $ the conditional variances are
defined as $\sigma _{k}^{2}\left( K\right) =$ Var$\left[ \mathbb{E}\left[
K\left( X_{1},...,X_{m}\right) |X_{1},...,X_{k}\right] \right] $, where the
dependence on $K$ is usually omitted.

\subsection{A brief history}

Hoeffding (\cite{hoeffding1948class} proved the following classical results
on the variance and the asymptotic behavior of U-statistics.

\begin{theorem}
	\label{Theorem Variance Hoeffding} If $\sigma _{m}^{2}<\infty $%
	\begin{eqnarray*}
		\text{Var}\left[ U\left( \mathbf{X}\right) \right] &=&\binom{n}{m}%
		^{-1}\sum_{k=1}^{m}\binom{m}{k}\binom{n-m}{k}\sigma _{k}^{2} \\
		\text{Var}\left[ U\left( \mathbf{X}\right) \right] &=&\frac{m^{2}\sigma
			_{1}^{2}}{n}+O\left( n^{-2}\right) \\
		n^{1/2}\left( U-\theta \right) &\rightarrow &\mathcal{N}\left( 0,m^{2}\sigma
		_{1}^{2}\right) \text{ in distribution as }n\rightarrow \infty \text{.}
	\end{eqnarray*}
\end{theorem}

In the degenerate case, when $\sigma _{1}^{2}=0$, the central limit theorem
given does not provide much information, one has to normalize with $n$
instead of $n^{1/2}$ and the limiting distribution will be chi-squared. We
will give no special consideration to the degenerate cases in this paper.
Hoeffding also gives the following Bernstein-type inequality for
U-statistics.

\begin{theorem}
	\label{Theorem Bernstein Hoeffding} \cite{hoeffding58probability} If $%
	\left\lfloor n/m\right\rfloor =n/m$ then $\forall t>0$%
	\begin{equation*}
	\Pr \left\{ U\left( \mathbf{X}\right) -\theta >t\right\} \leq \exp \left( 
	\frac{-nt^{2}}{2m\sigma _{m}^{2}+4mt/3}\right) 
	\end{equation*}
\end{theorem}

The assumption that $n$ is a multiple of $m$ is made for convenience. Let $%
M=n/m$ and $\mathbf{W}=\left( \left( X_{1},...,X_{m}\right) ,\left(
X_{m+1},...,X_{2m}\right) ,...,\left( X_{\left( R-1\right)
	m+1},...,X_{Rm}\right) \right) $. Then the incomplete U-statistics $U_{%
	\mathbf{W}}\left( \mathbf{X}\right) $ is a sum of independent variables to
which the standard form of Bernstein's inequality can be applied and yields
the above bound for $\Pr \left\{ U_{\mathbf{W}}\left( \mathbf{X}\right)
-\theta >t\right\} $. This is the basis of Hoeffding's proof and provides a
Bernstein-type inequality for incomplete U-statistics with budget $n/m$.
Hoeffding uses convexity of the moment generating function and the average
of all permutations of the $X_{i}$ to show that the bound holds also for the
complete U-statistic. This type of argument has found many applications in
particular to the expected suprema of U-processes, but we will not follow up
on it any further.

Above inequality also gives%
\begin{eqnarray*}
	\Pr \left\{ n^{1/2}\left( U\left( \mathbf{X}\right) -\theta \right)
	>t\right\}  &\leq &\exp \left( \frac{-t^{2}}{2m\sigma _{m}^{2}+4mn^{-1/2}/3}%
	\right)  \\
	&\rightarrow &\exp \left( \frac{-t^{2}}{2m\sigma _{m}^{2}}\right) \text{ as }%
	n\rightarrow \infty \text{.}
\end{eqnarray*}%
This does not match the CLT in Theorem \ref{Theorem Variance Hoeffding}. If
we take an inequality of Bernstein-type as a template, we would prefer%
\begin{equation*}
\Pr \left\{ U\left( \mathbf{X}\right) -\theta >t\right\} \leq \exp \left( 
\frac{-nt^{2}}{2m^{2}\sigma _{1}^{2}+o\left( n\right) +C(m)t}\right) .
\end{equation*}%
Such a result has been given by Arcones \cite{arcones1995bernstein}.

\begin{theorem}
	\label{Theorem Bernstein Arcones}$\forall t>0$ 
	\begin{equation*}
	\Pr \left\{ U\left( \mathbf{X}\right) -\theta >t\right\} \leq 2\exp \left( 
	\frac{-nt^{2}}{2m^{2}\sigma _{1}^{2}+\left( 2^{m+2}m^{m}+2m^{-1}/3\right) t}%
	\right) . 
	\end{equation*}
\end{theorem}

Apart from the leading constant $2$ this is consistent with Hoeffding's CLT
in Theorem \ref{Theorem Variance Hoeffding}. But from a practical point of
view the bound is useless for large values of $m$. Already for $m=5$ we
would need $n$ to be at least $10^{5}$ to give the right-hand-side a
nontrivial value less than one. For a complete U-statistic, this would
require at least $10^{25}$ kernel evaluations.

The proof of Theorem \ref{Theorem Bernstein Arcones} is based on the
consideration of decoupled U-statistics, where $\mathbf{X}$ is given an
independent copy for each argument of the kernel. The bounds for decoupled
U-statistics is then related to the original U-statistic with the help of
decoupling inequalities (De la Pe\~{n}a (1992)). While these techniques have
led to qualitatively very sharp bounds for degenerate U-statistics (\cite%
{arcones1993limit}, \cite{gine2000exponential}, \cite{adamczak2006moment}),
they are also responsible for the excessive size of the scale term in above
inequality.

Using a general concentration inequality for functions of independent
variables \cite{maurer2019bernstein} essentially gives the following bound
(slightly simplified and adapted to the scaling of $K$ chosen here).

\begin{theorem}
	\label{Theorem Maurer17} $\forall t>0$ 
	\begin{equation*}
	\Pr \left\{ U\left( \mathbf{X}\right) -\mathbb{\theta }>t\right\} \leq \exp
	\left( \frac{-t^{2}}{2\text{Var}\left( U\left( \mathbf{X}\right) \right)
		+8m^{2}/n^{2}+4\left( m^{2}+m/3\right) t/n}\right) . 
	\end{equation*}
\end{theorem}

This bound is consistent with the CLT and avoids the exponential dependence
of the scale term on $m$. Here we give a similar inequality for incomplete
U-statistics. We show that it leads to a refinement of Theorem \ref{Theorem
	Maurer17} for complete U-statistics, and that it suffices to make $n^{2}$
kernel evaluations, regardless of the degree $m$ of the kernel to get the
same order bound.

The study of incomplete U-statistics seems to begin with the work of Blom 
\cite{blom1976some}. This and several other works considered the question of
choosing $\mathbf{W}$ so as to minimize the variance of the incomplete
statistic. \cite{blom1976some} and Lee \cite{lee1982incomplete} consider
balanced incomplete block designs, which require $M=n^{2}$ for any order of
the kernel. Kong et al (\cite{kong2020design}) propose a sophisticated
design strategy, which also takes the values of the $X_{i}$ into account.
Their method is shown to be asymptotically efficient in the sense, that the
quotients of the variances of the corresponding incomplete U-statistics and
the complete U-statistic approaches one as the sample size $n$ tends to
infinity. All these results are either asymptotic, or they only give
probabilistic guarantees via Chebychev's inequality, while exponential
bounds should be possible for bounded kernels.

Several authors (\cite{clemenccon2013maximal}, \cite{clemenccon2016scaling}, 
\cite{chen2019randomized}, \cite{chen2019randomized}) consider the
optimization of incomplete U-statistics over classes of kernels. These works
give uniform finite sample bounds, but they are of the worst-case type and
do not take variance information into account.\bigskip 

\section{Results\label{Section Results}}

The concentration properties of a statistic depend on its sensitivity to the
modification of a small portion of the data. To state our results we define
the relevant sensitivity properties of the design, and, separate from this,
of the kernel.

For an incomplete U-statistic with design $\mathbf{W}=\left(
W_{1},...,W_{M}\right) $ the modification of a datum $X_{k}$ will only
affect those kernel evaluations $K\left( \mathbf{X}^{W_{i}}\right) $, for
which $k\in W_{i}$. This motivates the following definition.

\begin{definition}
	For a design $\mathbf{W}=\left( W_{1},...,W_{M}\right) \in \left\{ W\subset %
	\left[ n\right] :\left\vert W\right\vert =m\right\} ^{M}$ and $k,l\in
	\left\{ 1,...,n\right\} $, $k\neq l$ define%
	\begin{gather*}
	R_{k}\left( \mathbf{W}\right) \triangleq \left\vert \left\{ i:k\in
	W_{i}\right\} \right\vert \text{ and }R_{k,l}\left( \mathbf{W}\right)
	\triangleq \left\vert \left\{ i:k,l\in W_{i}\right\} \right\vert  \\
	A\left( \mathbf{W}\right) \triangleq \sum_{k=1}^{n}\frac{R_{k}^{2}\left( 
		\mathbf{W}\right) }{M^{2}}\text{, \ \ \ \ \ \ \ \ \ \ \ }B\left( \mathbf{W}%
	\right) \triangleq \sum_{k,l:k\neq l}\frac{R_{kl}^{2}\left( \mathbf{W}%
		\right) }{M^{2}}\text{ } \\
	\text{and }C\left( \mathbf{W}\right) \triangleq \max_{k}\frac{R_{k}}{M}.
	\end{gather*}%
	When there is no ambiguity we omit the dependency on the design $\mathbf{W}$.
\end{definition}

The crucial properties of the kernel are its response to mixed partial
difference operations.

\begin{definition}
	For a bounded kernel $K:\mathcal{X}^{m}\rightarrow \left[ -1,1\right] $ and $%
	X_{1},...,X_{m},X_{1}^{\prime },X_{2}^{\prime }$ iid with values in $%
	\mathcal{X}$ define%
	\begin{eqnarray*}
		\beta \left( K\right)  &=&\mathbb{E}\left[ \left( D_{X_{1},X_{1}^{\prime
			}}^{1}D_{X_{2},X_{2}^{\prime }}^{2}K\left( X_{1},X_{2},...,X_{m}\right)
			\right) ^{2}\right]  \\
			\gamma \left( K\right)  &=&\sup_{\mathbf{x}\in \mathcal{X}^{m},y,y^{\prime
				}\in \mathcal{X}}\mathbb{E}\left[ \left( D_{y,y^{\prime
			}}^{1}D_{X_{2},X_{2}^{\prime }}^{2}K\left(
			x_{1},x_{2},x_{3},...,x_{m}\right) \right) ^{2}\right]  \\
			\alpha \left( K\right)  &=&\left( \sqrt{\beta \left( K\right) /2}+\sqrt{%
				\gamma \left( K\right) }\right) .
		\end{eqnarray*}%
		When there is no ambiguity we omit the dependency on the kernel $K$.
	\end{definition}
	
	Worst-case bounds for $\beta $, $\gamma $ and $\alpha $ are $\beta \leq 8$, $%
	\gamma \leq 8$ and $\alpha \leq 2+\sqrt{8}\leq 5$ (see Lemma \ref{Lemma
		bound Beta and Gamma} below).
	
	The following is our main Bernstein-type inequality for incomplete
	U-statistics.\bigskip
	
	\begin{theorem}
		\label{Theorem incomplete U}For fixed kernel $K$ with values in $\left[ -1,1%
		\right] $, design $\mathbf{W}$ and $t>0$%
		\begin{equation*}
		\Pr \left\{ U_{\mathbf{W}}\left( \mathbf{X}\right) -\mathbb{\theta }%
		>t\right\} \leq \exp \left( \frac{-t^{2}}{2A\sigma _{1}^{2}+B\beta /2+\left( 
			\sqrt{B\gamma }+4C/3\right) t}\right) , 
		\end{equation*}%
		and for $0<\delta \leq 1/e$ with probability at least $1-\delta $%
		\begin{equation*}
		U_{\mathbf{W}}\left( \mathbf{X}\right) -\theta \leq \sqrt{2A\sigma
			_{1}^{2}\ln \left( 1/\delta \right) }+\left( \alpha \sqrt{B}+4C/3\right) \ln
		\left( 1/\delta \right) . 
		\end{equation*}
	\end{theorem}
	
	The proof will be given in the next section. Since $\alpha $, $\beta $ and $%
	\gamma $ are invariant under $K\longleftrightarrow -K$ the same bound holds
	for $\theta -U_{\mathbf{W}}\left( \mathbf{X}\right) $
	
	As a first application we give a refinement of Theorem \ref{Theorem Maurer17}%
	.
	
	\begin{corollary}
		\label{Corollary Complete U-statistic}For $K$ with values in $\left[ -1,1%
		\right] $ and $t>0$%
		\begin{equation*}
		\Pr \left\{ U\left( \mathbf{X}\right) -\mathbb{\theta }>t\right\} \leq \exp
		\left( \frac{-nt^{2}}{2m^{2}\sigma _{1}^{2}+m^{4}\beta /\left( 2n\right)
			+\left( m^{2}\sqrt{\gamma }+\left( 4/3\right) m\right) t}\right) ,
		\end{equation*}%
		and for $0<\delta \leq 1/e$ with probability at least $1-\delta $%
		\begin{equation}
		U\left( \mathbf{X}\right) -\theta \leq \sqrt{\frac{2m^{2}\sigma _{1}^{2}\ln
				\left( 1/\delta \right) }{n}}+\frac{\left( \alpha m^{2}+4m/3\right) \ln
			\left( 1/\delta \right) }{n}.  \label{Confidence boud complete U-statistics}
		\end{equation}%
		\bigskip 
	\end{corollary}
	
	\begin{proof}
		For U-statistics $M=\binom{n}{m}$, $R_{k}=\binom{n-1}{m-1}$ and for $k\neq
		l,R_{kl}=\binom{n-2}{m-2}$. Thus $A=m^{2}/n$, $B=\frac{m^{2}\left(
			m-1\right) ^{2}}{n\left( n-1\right) }\leq \frac{m^{4}}{n^{2}}$ and $C=m/n$.
		Substitute in Theorem \ref{Theorem incomplete U}.
	\end{proof}
	
	This result is a slight improvement of Theorem \ref{Theorem Maurer17} in two
	ways. First since 
	\begin{equation*}
	\frac{m^{2}\sigma _{1}^{2}}{n}\leq \text{Var}\left( U\left( \mathbf{X}%
	\right) \right) 
	\end{equation*}%
	and second since the coefficients $\beta $ and $\gamma $ can be
	substantially smaller than their worst-case bounds. If, for example $%
	\mathcal{X}$ is a metric space and the kernel is separately Lipschitz $L$ in
	each argument. Then $\beta $ and $\gamma $ can be bounded by $L$Var$\left(
	X_{1}\right) $. Moreover, if mixed second partial differences of the kernel
	are of $O\left( m^{-1}\right) $, then $\alpha =O\left( m^{-1}\right) $ and
	meaningful bounds result for U-statistics of growing order as long as $%
	m=o\left( n\right) $. 
	
	For comparison to the bound of Arcones (Theorem \ref{Theorem Bernstein
		Arcones}) we substitute the worst case values for $\beta $ and $\gamma $ in
	Corollary \ref{Corollary Complete U-statistic} to obtain 
	\begin{equation}
	\Pr \left\{ U\left( \mathbf{X}\right) -\mathbb{\theta }>t\right\} \leq \exp
	\left( \frac{-nt^{2}}{2m^{2}\sigma _{1}^{2}+4m^{4}/n+\left( \sqrt{8}%
		m^{2}+\left( 4/3\right) m\right) t}\right)  \label{Worst Case Complete}
	\end{equation}%
	and assume $\sigma _{1}^{2}=0$. If the bound in Theorem \ref{Theorem
		Bernstein Arcones} was less than or equal to (\ref{Worst Case Complete}) it
	would have to be nontrivial, so $nt^{2}>\left( 2m^{2}\sigma _{1}^{2}+\left(
	2^{m+2}m^{m}+2m^{-1}/3\right) t\right) \ln 2$, which implies $%
	nt>2^{m+2}m^{m}\ln 2$ On the other hand, even ignoring the factor 2 in
	Theorem \ref{Theorem Bernstein Arcones}, we would also need%
	\begin{equation*}
	\frac{4m^{4}}{\left( 2^{m+2}m^{m}+2m^{-1}/3-\left( m^{2}\sqrt{8}+\left(
		4/3\right) m\right) \right) }\geq nt. 
	\end{equation*}%
	This implies 
	\begin{equation*}
	\frac{m^{4}}{\left( 2^{m+2}m^{m}-\left( m^{2}\sqrt{8}+\left( 4/3\right)
		m\right) \right) }\geq 2^{m}m^{m}\ln 2, 
	\end{equation*}%
	which is easily seen to be false for all positive integers $m$. The bound in
	Corollary \ref{Corollary Complete U-statistic} is therefore smaller than the
	one in Theorem \ref{Theorem Bernstein Arcones} for all values of $m,n\in 
	\mathbb{N}$ and $t>0$.
	
	The second application of Theorem \ref{Theorem incomplete U} concerns random
	designs, where the design $\mathbf{W}$ is sampled with replacement from the
	uniform distribution on the set $\left\{ W\subset \left[ n\right]
	:\left\vert W\right\vert =m\right\} $, independent of $\mathbf{X}$. \bigskip 
	
	\begin{theorem}
		\label{Theorem random design}Under random sampling of $\mathbf{W}=\left(
		W_{1},...,W_{M}\right) $ if $M\geq \ln ^{2}n$, then with $\delta _{i}>0$ and
		probability at least $1-\left( \delta _{1}+\delta _{2}\right) $%
		\begin{eqnarray*}
			U_{\mathbf{W}}\left( \mathbf{X}\right) -\theta  &\leq &\sqrt{\frac{%
					2m^{2}\sigma _{1}^{2}\ln \left( 1/\delta _{1}\right) }{n}}+\frac{\alpha
				m^{2}+\left( 4/3\right) m}{n}\ln \left( 1/\delta _{1}\right)  \\
			&&\text{ \ \ \ \ \ \ \ \ \ }+\frac{5\alpha m+9\sqrt{m}+4}{\sqrt{M}}\ln
			^{2}\left( 3/\delta _{2}\right) .
		\end{eqnarray*}
	\end{theorem}
	
	\textbf{Remarks:} 1. The first two terms of the bound match the bound (\ref%
	{Confidence boud complete U-statistics}) for complete U-statistics given
	above, apart from the factor $\ln \left( 1/\delta _{1}\right) $ instead of $%
	\ln \left( 1/\delta \right) $, which arises from a union bound. The
	remaining terms bound the error incurred by incompleteness.
	
	2. If $M=O\left( n^{1+\epsilon }\right) $ for any $\epsilon >0$, let $\delta
	_{1}=\left( 1-n^{-\epsilon }\right) \delta $ and $\delta _{2}=n^{-\epsilon
	}\delta $, with $\delta =\exp \left( -t^{2}/\left( 2m^{2}\sigma
	_{1}^{2}\right) \right) $. Then as $n\rightarrow \infty $%
	\begin{equation*}
	\Pr \left\{ \sqrt{n}\left( U_{\mathbf{W}}\left( \mathbf{X}\right) -\theta
	\right) >t\right\} \rightarrow \exp \left( \frac{-t^{2}}{2m^{2}\sigma
		_{1}^{2}}\right) .
	\end{equation*}%
	For random designs with computational budget $M=n^{1+\epsilon }$ the bound
	is consistent with the CLT for the complete U-statistic.
	
	3. For $M=n^{2}$ we recover the order $1/n$ of the subexponential term of
	the complete statistic in Corollary \ref{Corollary Complete U-statistic}.
	This is no help for kernels of order 2, as they frequently occur in
	applications, but for kernels of order 3 it already gives a significant
	advantage.\bigskip 
	
	For positive bounded kernels we give a sub-Gaussian bound on the lower tail,
	which can be used to obtain empirical, variance-dependent bounds.
	
	\begin{theorem}
		\label{Theorem SubGauss}For fixed kernel $K$ with values in $\left[ 0,1%
		\right] $, design $\mathbf{W}$ and $t>0$%
		\begin{equation*}
		\Pr \left\{ \mathbb{E}\left[ U_{\mathbf{W}}\left( \mathbf{X}\right) \right]
		-U_{\mathbf{W}}\left( \mathbf{X}\right) >t\right\} \leq \exp \left( \frac{%
			-t^{2}}{8mC~\mathbb{E}\left[ U_{\mathbf{W}}\left( \mathbf{X}\right) \right] }%
		\right) .
		\end{equation*}%
		For $0<\delta $ we have%
		\begin{equation*}
		\Pr \left\{ \sqrt{\mathbb{E}\left[ U_{\mathbf{W}}\left( \mathbf{X}\right) %
			\right] }>\sqrt{U_{\mathbf{W}}\left( \mathbf{X}\right) }+\sqrt{8mC\ln \left(
			1/\delta \right) }\right\} \leq \delta \text{.}
		\end{equation*}%
		These inequalities do not require permutation invariance of the kernel.
	\end{theorem}
	
	\section{Proofs\label{Section Proofs}}
	
	In this section we prove the above results. Necessary auxiliary results are
	introduced as needed.
	
	\subsection{Incomplete U-statistics\protect\bigskip }
	
	\begin{definition}
		For $f:\mathcal{X}^{n}\rightarrow \mathbb{R}$ define%
		\begin{eqnarray*}
			ES_{\mathbf{X}}\left( f\right)  &:&=\frac{1}{2}\sum_{k=1}^{n}\mathbb{E}\left[
			\left( D_{X_{k},X_{k}^{\prime }}^{k}f\left( \mathbf{X}\right) \right) ^{2}%
			\right] \text{ and} \\
			H_{\mathbf{X}}\left( f\right)  &:&=\sum_{k,l:k\neq l}\mathbb{E}\left[ \left(
			D_{X_{l},X_{l}^{\prime }}^{l}D_{X_{k},X_{k}^{\prime }}^{k}f\left( \mathbf{X}%
			\right) \right) ^{2}\right] .
		\end{eqnarray*}
	\end{definition}
	
	The first of these is the Efron-Stein estimate of the variance. These
	quantities are related to the variance of $f\left( \mathbf{X}\right) $ and a
	lower bound on the variance as follows.
	
	\begin{theorem}
		\label{Theorem Variance inequalities}\cite{Efron81}, \cite{Houdre97}: 
		\begin{multline*}
		\sum_{k=1}^{n}\text{Var}\left[ \mathbb{E}\left[ f\left( \mathbf{X}\right)
		|X_{k}\right] \right] \leq \text{Var}\left[ f\left( \mathbf{X}\right) \right]
		\leq ES_{\mathbf{X}}\left( f\right)  \\
		\leq \sum_{k=1}^{n}\text{Var}\left[ \mathbb{E}\left[ f\left( \mathbf{X}%
		\right) |X_{k}\right] \right] +\frac{1}{4}H_{\mathbf{X}}\left( f\right) ,
		\end{multline*}
	\end{theorem}
	
	The second inequality is the well-known Efron-Stein inequality. The above
	chain of inequalities bounds the bias of the Efron-Stein estimate. If $f$ is
	a sum of independent component functions, then $D_{X_{l},X_{l}^{\prime
		}}^{l}D_{X_{k},X_{k}^{\prime }}^{k}f\left( \mathbf{X}\right) $ is almost
		surely zero, so $H_{\mathbf{X}}\left( f\right) $ vanishes and all the
		inequalities become identities. We also need the following concentration
		inequality from \cite{maurer2019bernstein}.
		
		\begin{theorem}
			\label{Theorem general Bernstein} For $f:\mathcal{X}^{n}\rightarrow \mathbb{R%
			}$, if $\forall k,f\left( \mathbf{X}\right) -\mathbb{E}\left[ f\left( 
			\mathbf{X}\right) |\mathbf{X}^{\backslash k}\right] <b$ then for $t>0$%
			\begin{equation*}
			\Pr \left\{ f\left( \mathbf{X}\right) -\mathbb{E}\left[ f\left( \mathbf{X}%
			\right) \right] >t\right\} \leq \exp \left( \frac{-t^{2}}{2ES_{\mathbf{X}%
				}\left( f\right) +\left( J_{\mathbf{X}}\left( f\right) +2b/3\right) t}%
			\right) ,
			\end{equation*}%
			where $J_{\mathbf{X}}\left( f\right) $ is the interaction-functional%
			\begin{equation*}
			J_{\mathbf{X}}\left( f\right) :=\left( \sup_{\mathbf{x}\in \mathcal{X}%
				^{n}}\sum_{k,l:k\neq l}\sup_{y,y^{\prime }\in \mathcal{X}}\mathbb{E}\left[
			\left( D_{y,y^{\prime }}^{l}D_{X_{k},X_{k}^{\prime }}^{k}f\left( \mathbf{x}%
			\right) \right) ^{2}\right] \right) ^{1/2}.
			\end{equation*}
		\end{theorem}
		
		Again, if $f$ is a sum then $J_{\mathbf{X}}$ vanishes and $ES_{\mathbf{X}%
		}\left( f\right) $ becomes equal to Var$\left[ f\left( \mathbf{X}\right) %
		\right] $, so the inequality reduces to the classical Bernstein inequality
		for sums (e.g. \cite{McDiarmid98}). For more general functions, however, $%
		ES_{\mathbf{X}}\left( f\right) $ overestimates the variance, so the
		inequality is not quite a proper Bernstein inequality. Theorem \ref{Theorem
			Variance inequalities}\ above allows us to bound this overestimation. These
		results and the next simple lemma provide a proof of Theorem \ref{Theorem
			incomplete U}.\bigskip 
		
		\begin{lemma}
			\label{Lemma Jensen}For $N\in \mathbb{N}$, $i\in \left[ N\right] $ let $%
			F_{i}:\mathcal{Y}\times \mathcal{Z}\rightarrow \mathbb{R}$ with $Z$ a random
			variable with values in $\mathcal{Z}$. Then%
			\begin{equation*}
			\sup_{y}\mathbb{E}\left[ \left( \sum_{i=1}^{N}F_{i}\left( y,Z\right) \right)
			^{2}\right] \leq N^{2}\max_{i}\sup_{y}\mathbb{E}\left[ F_{i}\left(
			y,Z\right) ^{2}\right] . 
			\end{equation*}
		\end{lemma}
		
		\begin{proof}
			With Jensen's inequality%
			\begin{eqnarray*}
				\sup_{y}\mathbb{E}\left[ \left( \sum_{i=1}^{N}F_{i}\left( y,Z\right) \right)
				^{2}\right] &=&N^{2}\sup_{y}\mathbb{E}\left[ \left( \frac{1}{N}%
				\sum_{i=1}^{N}F_{i}\left( y,Z\right) \right) ^{2}\right] \\
				&\leq &N^{2}\sup_{y}\frac{1}{N}\sum_{i=1}^{N}\mathbb{E}\left[ F_{i}\left(
				y,Z\right) ^{2}\right] \\
				&\leq &N^{2}\max_{i}\sup_{y}\mathbb{E}\left[ F_{i}\left( y,Z\right) ^{2}%
				\right] .
			\end{eqnarray*}%
			\bigskip
		\end{proof}
		
		\begin{proof}[Proof of Theorem \protect\ref{Theorem incomplete U}]
			By the bound on $K$ we have for all $k$%
			\begin{eqnarray}
			U_{\mathbf{W}}\left( \mathbf{X}\right) -\mathbb{E}\left[ U_{\mathbf{W}%
			}\left( \mathbf{X}\right) |\mathbf{X}^{\backslash k}\right]  &=&\frac{1}{M}%
			\sum_{i=1}^{M}K\left( \mathbf{X}^{W_{i}}\right) -\mathbb{E}\left[ K\left( 
			\mathbf{X}^{W_{i}}\right) |\mathbf{X}^{\backslash k}\right]   \notag \\
			&=&\frac{1}{M}\sum_{i:k\in W_{i}}K\left( \mathbf{X}^{W_{i}}\right) -\mathbb{E%
			}\left[ K\left( \mathbf{X}^{W_{i}}\right) |\mathbf{X}^{\backslash k}\right] 
			\notag \\
			&\leq &\frac{2R_{k}}{M}\leq 2C.  \label{Boundonb}
			\end{eqnarray}%
			If $\left\vert W\right\vert =m$, $k\neq l$ and not both $k$ and $l$ are
			members of $W$ then $D_{z,z^{\prime }}^{l}D_{y,y^{\prime }}^{k}K\left( 
			\mathbf{x}^{W}\right) =0$. Thus for any $k,l\in \left[ n\right] $, $k\neq l$ 
			\begin{align*}
			& \sup_{\mathbf{x}\in \mathcal{X}^{n},y,y^{\prime }\in \mathcal{X}}\mathbb{E}%
			\left[ \left( \sum_{i=1}^{M}D_{y,y^{\prime }}^{l}D_{X_{k},X_{k}^{\prime
				}}^{k}K\left( \mathbf{x}^{W_{i}}\right) \right) ^{2}\right]  \\
				& =\sup_{\mathbf{x}\in \mathcal{X}^{n},y,y^{\prime }\in \mathcal{X}}\mathbb{E%
				}\left[ \left( \sum_{i:\left\{ k,l\right\} \in W_{i}}D_{y,y^{\prime
				}}^{l}D_{X_{k},X_{k}^{\prime }}^{k}K\left( \mathbf{x}^{W_{i}}\right) \right)
				^{2}\right]  \\
				& \leq \left( R_{kl}\right) ^{2}\max_{i:\left\{ k,l\right\} \in W_{i}}\sup_{%
					\mathbf{x}\in \mathcal{X}^{n},y,y^{\prime }\in \mathcal{X}}\mathbb{E}\left[
				\left( D_{y,y^{\prime }}^{l}D_{X_{k},X_{k}^{\prime }}^{k}K\left( \mathbf{x}%
				^{W_{i}}\right) \right) ^{2}\right]  \\
				& \leq \left( R_{kl}\right) ^{2}\max_{k\neq l}\max_{i:\left\{ k,l\right\}
					\in W_{i}}\sup_{\mathbf{x}\in \mathcal{X}^{n},y,y^{\prime }\in \mathcal{X}}%
				\mathbb{E}\left[ \left( D_{y,y^{\prime }}^{l}D_{X_{k},X_{k}^{\prime
					}}^{k}K\left( \mathbf{x}^{W_{i}}\right) \right) ^{2}\right]  \\
					& \leq \left( R_{kl}\right) ^{2}\max_{i}\max_{k,l\in \left[ m\right] ,k\neq
						l}\sup_{\mathbf{x}\in \mathcal{X}^{m},y,y^{\prime }\in \mathcal{X}}\mathbb{E}%
					\left[ \left( D_{y,y^{\prime }}^{l}D_{X_{k},X_{k}^{\prime }}^{k}K\left( 
					\mathbf{x}\right) \right) ^{2}\right]  \\
					& =\left( R_{kl}\right) ^{2}\gamma .
					\end{align*}%
					The first inequality follows Lemma \ref{Lemma Jensen}. The last step follows
					from permutation symmetry of the kernel and the definition of $\gamma $.
					Thus 
					\begin{eqnarray*}
						J_{\mathbf{X}}^{2}\left( U_{\mathbf{W}}\right)  &=&\frac{1}{M^{2}}\sup_{%
							\mathbf{x}\in \mathcal{X}^{n}}\sum_{k,l:k\neq l}\sup_{y,y^{\prime }\in 
							\mathcal{X}}\mathbb{E}\left[ \left( \sum_{i=1}^{M}D_{y,y^{\prime
							}}^{l}D_{X_{k},X_{k}^{\prime }}^{k}K\left( \mathbf{x}^{W_{i}}\right) \right)
							^{2}\right]  \\
							&\leq &\frac{1}{M^{2}}\sum_{k,l:k\neq l}\sup_{\mathbf{x}\in \mathcal{X}%
								^{n}}\sup_{y,y^{\prime }\in \mathcal{X}}\mathbb{E}\left[ \left(
							\sum_{i=1}^{M}D_{y,y^{\prime }}^{l}D_{X_{k},X_{k}^{\prime }}^{k}K\left( 
							\mathbf{x}^{W_{i}}\right) \right) ^{2}\right]  \\
							&\leq &\frac{1}{M^{2}}\sum_{k,l:k\neq l}R_{kl}^{2}\gamma =B\gamma .
						\end{eqnarray*}%
						In exactly the same way one proves that $H_{\mathbf{X}}\left( U_{\mathbf{W}%
						}\right) \leq B\beta .$
						
						From the general Bernstein inequality Theorem \ref{Theorem general Bernstein}%
						, (\ref{Boundonb}) and the bound on $J_{\mathbf{X}}^{2}\left( U_{\mathbf{W}%
						}\right) $ we obtain%
						\begin{equation}
						\Pr \left\{ U_{\mathbf{W}}\left( \mathbf{X}\right) -\mathbb{E}\left[ U_{%
							\mathbf{W}}\left( \mathbf{X}\right) \right] >t\right\} \leq \exp \left( 
						\frac{-t^{2}}{2ES_{\mathbf{X}}\left( U_{\mathbf{W}}\right) +\left( \sqrt{%
								B\gamma }+4C/3\right) t}\right) .  \label{Intermediate B for incomplete}
						\end{equation}
						
						It remains to bound the Efron-Stein term $ES_{\mathbf{X}}\left( U_{\mathbf{W}%
						}\right) $. Again using Lemma \ref{Lemma Jensen}%
						\begin{eqnarray*}
							\text{Var}\left[ \sum_{i=1}^{M}\mathbb{E}\left[ K\left( \mathbf{X}%
							^{W_{i}}\right) |X_{k}\right] \right]  &=&\frac{1}{2}\mathbb{E}\left[ \left(
							\sum_{i:k\in W_{i}}\mathbb{E}\left[ K\left( \mathbf{X}^{W_{i}}\right) |X_{k}%
							\right] -\mathbb{E}\left[ K\left( \mathbf{X}^{\prime W_{i}}\right)
							|X_{k}^{\prime }\right] \right) ^{2}\right]  \\
							&\leq &\frac{R_{k}^{2}}{2}\max_{i}\mathbb{E}\left[ \left( \mathbb{E}\left[
							K\left( \mathbf{X}^{W_{i}}\right) |X_{k}\right] -\mathbb{E}\left[ K\left( 
							\mathbf{X}^{\prime W_{i}}\right) |X_{k}^{\prime }\right] \right) ^{2}\right] 
							\\
							&=&R_{k}^{2}\sigma _{1}^{2}.
						\end{eqnarray*}%
						The last step follows since $\mathbb{E}\left[ K\left( \mathbf{X}%
						^{W_{i}}\right) |X_{k}\right] $ is identically distributed to $\mathbb{E}%
						\left[ K\left( X_{1},...,X_{m}\right) |X_{1}\right] $, so%
						\begin{equation}
						\sum_{k=1}^{n}\text{Var}\left[ \mathbb{E}\left[ U_{\mathbf{W}}\left( \mathbf{%
							X}\right) |X_{k}\right] \right] =\frac{1}{M^{2}}\sum_{k=1}^{n}\text{Var}%
						\left[ \sum_{i=1}^{M}\mathbb{E}\left[ K\left( \mathbf{X}^{W_{i}}\right)
						|X_{k}\right] \right] \leq A\sigma _{1}^{2}\text{.}  \label{A-bound}
						\end{equation}%
						\newline
						From Theorem \ref{Theorem Variance inequalities}, (\ref{A-bound}) and the
						bound on $H_{\mathbf{X}}\left( U_{\mathbf{W}}\right) $ we get%
						\begin{eqnarray*}
							ES_{\mathbf{X}}\left( U_{\mathbf{W}}\right)  &\leq &\sum_{k=1}^{n}\text{Var}%
							\left[ \mathbb{E}\left[ U_{\mathbf{W}}\left( \mathbf{X}\right) |X_{k}\right] %
							\right] +\frac{1}{4}H_{\mathbf{X}}\left( U_{\mathbf{W}}\right)  \\
							&\leq &A\sigma _{1}^{2}+B\beta /4.
						\end{eqnarray*}%
						Substitution in (\ref{Intermediate B for incomplete}) completes the proof of
						the first inequality. The second assertion follows from equating the bound
						on the probability to $\delta $, solving for $t$ and using the fact that $%
						0\leq \delta \leq 1/e$ implies $\sqrt{\ln \left( 1/\delta \right) }\leq \ln
						\left( 1/\delta \right) $.
					\end{proof}
					
					\subsection{Random design}
					
					In this section we assume that $\mathbf{W}=\left( W_{1},...,W_{M}\right) $
					is sampled independent of $\mathbf{X}$ and with replacement from the uniform
					distribution on $\left\{ W\subset \left\{ 1,...,n\right\} :\left\vert
					W\right\vert =m\right\} $. The quantities $A\left( \mathbf{W}\right) $, $%
					B\left( \mathbf{W}\right) $ and $C\left( \mathbf{W}\right) $ are now random
					variables.
					
					For $k,l\in \left\{ 1,...,n\right\} $ and $k\neq l$ define for $i\in \left[ M%
					\right] $ the Bernoulli variables $Z_{i}^{k}=1_{\left\{ k\in W_{i}\right\} }$
					and $Z_{i}^{kl}=\mathbf{1}_{\left\{ k,l\in W_{i}\right\} }$. For fixed $k$
					and $l$ the $Z_{i}^{k}$ are iid variables and so are the $Z_{i}^{kl}$. It is
					easy to see that $\mathbb{E}\left[ Z_{i}^{k}\right] =m/n$ and $\mathbb{E}%
					\left[ Z_{i}^{kl}\right] =m\left( m-1\right) /\left( n\left( n-1\right)
					\right) $. Also $R^{k}=\sum_{i}Z_{i}^{k}$ and $R^{kl}=\sum_{i}Z_{i}^{kl}$.
					Then%
					\begin{gather*}
					A\left( \mathbf{W}\right) =\frac{1}{M^{2}}\sum_{k=1}^{n}\left(
					\sum_{i=1}^{M}Z_{i}^{k}\right) ^{2}\text{ \ \ \ \ }B\left( \mathbf{W}\right)
					=\frac{1}{M^{2}}\sum_{\substack{ k,l=1 \\ k\neq l}}^{n}\left(
					\sum_{i=1}^{M}Z_{i}^{kl}\right) ^{2} \\
					C\left( \mathbf{W}\right) =\frac{1}{M}\max_{k=1}^{n}\sum_{i=1}^{M}Z_{i}^{k}
					\end{gather*}%
					To prove Theorem \ref{Theorem random design} we will give high probability
					bounds for $A$, $B$ and $C$.
					
					\begin{lemma}
						\label{Lemma C-bound}If $\sqrt{M}\geq \ln n$, then for $\delta \in \left(
						0,1\right) $ with probability at least $1-\delta /4$%
						\begin{equation*}
						C\left( \mathbf{W}\right) \leq \frac{m}{n}+\frac{\sqrt{2m}+3}{\sqrt{M}}\ln
						\left( 4/\delta \right) .
						\end{equation*}
					\end{lemma}
					
					\begin{proof}
						The variance of $Z_{i}^{k}$ is less than its expectation, which is $m/n$.
						Therefore by Bernstein's inequality for fixed $k\in \left[ n\right] $ with
						probability at least $1-\delta /4$%
						\begin{equation*}
						\frac{1}{M}\sum_{i=1}^{M}Z_{i}^{k}\leq \frac{m}{n}+\sqrt{\frac{2m\ln \left(
								4/\delta \right) }{nM}}+\frac{2\ln \left( 4/\delta \right) }{3M}.
						\end{equation*}%
						\begin{eqnarray*}
							\Pr \left\{ \frac{1}{M}\sum_{i=1}^{M}Z_{i}^{k}>\frac{m}{n}+t\right\}  &\leq
							&\exp \left( \frac{-Mt^{2}}{2m/n+2t/3}\right)  \\
							\Pr \left\{ C\left( \mathbf{W}\right) >\frac{m}{n}+t\right\}  &\leq &n\exp
							\left( \frac{-Mt^{2}}{2m/n+2t/3}\right) 
						\end{eqnarray*}%
						A union bound over $k\in \left[ n\right] $ gives with probability at least $%
						1-\delta /4$%
						\begin{eqnarray*}
							C\left( \mathbf{W}\right)  &\leq &\frac{m}{n}+\sqrt{\frac{2m\ln \left(
									4n/\delta \right) }{nM}}+\frac{2\ln \left( 4n/\delta \right) }{3M} \\
							&=&\frac{m}{n}+\sqrt{\frac{2m\ln n}{nM}}+\sqrt{\frac{2m\ln \left( 4/\delta
									\right) }{nM}}+\frac{2\ln n}{3M}+\frac{2\ln \left( 4/\delta \right) }{3M} \\
							&\leq &\frac{m}{n}+\frac{\sqrt{2m}+3}{\sqrt{M}}\ln \left( 4/\delta \right) ,
						\end{eqnarray*}%
						where the last inequality follows from $n\geq \ln n$, $n\geq m$, $\sqrt{M}%
						\geq \ln n$ and $\ln \left( 4/\delta \right) \geq 1$.\bigskip 
					\end{proof}
					
					Bounds for $A$ and $B$ could easily be given by bounding $%
					R_{k}=\sum_{i}Z_{i}^{k}$ (or $R_{kl}=\sum_{i}Z_{i}^{kl}$) with high
					probability, extending to all $k$ (or all pairs $k,l$) with a union bound
					and summing the squares. Unfortunately the union bound would incur a
					logarithmic factor in $n$. To avoid this we estimate the expectations of $A$
					and $B$ and the probability to deviate from these expectations.
					
					\begin{lemma}
						\label{Lemma expectations}%
						\begin{eqnarray*}
							\mathbb{E}\left[ A\left( \mathbf{W}\right) \right]  &\leq &\frac{m^{2}}{n}+%
							\frac{m}{M}\text{ and} \\
							\mathbb{E}\left[ B\left( \mathbf{W}\right) \right]  &\leq &\frac{m^{4}}{n^{2}%
							}+\frac{m^{2}}{M}
						\end{eqnarray*}%
						\bigskip 
					\end{lemma}
					
					\begin{proof}
						Using independence%
						\begin{eqnarray*}
							\mathbb{E}\left[ \left( \sum_{i=1}^{M}Z_{i}^{k}\right) ^{2}\right]
							&=&\sum_{i}\mathbb{E}\left[ Z_{i}^{k}\right] +\sum_{i\neq j}\mathbb{E}\left[
							Z_{i}^{k}Z_{j}^{k}\right] \\
							&=&\frac{Mm}{n}+\frac{M\left( M-1\right) m^{2}}{n^{2}}.
						\end{eqnarray*}%
						The first identity follows from multiplication with $n/M^{2}$. The second
						one follows similarly using $\mathbb{E}\left[ Z_{i}^{kl}\right] =\frac{%
							m\left( m-1\right) }{n\left( n-1\right) }$.\bigskip
					\end{proof}
					
					To give high-probability deviation bounds for $A$ and $B$ we now use a
					concentration inequality from \cite{maurer2006concentration}, which is
					particularly well suited for expressions involving the squares of random
					variables. For any $\mathcal{X}$ and functions $f:\mathcal{X}^{N}\rightarrow 
					\mathbb{R}$ define an operator $D^{2}$ by 
					\begin{equation*}
					D^{2}f\left( \mathbf{x}\right) =\sum_{k=1}^{N}\left( f\left( \mathbf{x}%
					\right) -\inf_{y\in \mathcal{X}}S_{y}^{k}f\left( \mathbf{x}\right) \right)
					^{2}.
					\end{equation*}
					
					\begin{theorem}
						\label{Theorem selfbound}Suppose $f:\mathcal{X}^{N}\rightarrow \mathbb{R}$
						satisfies for some $a>0$%
						\begin{equation}
						D^{2}f\left( \mathbf{x}\right) \leq af\left( \mathbf{x}\right) ,\forall 
						\mathbf{x}\in \mathcal{X}^{N}\text{,}  \label{selfbound condition}
						\end{equation}%
						and let $\mathbf{X}=\left( X_{1},...,X_{N}\right) $ be a vector of
						independent variables. Then for all $t>0$%
						\begin{equation*}
						\Pr \left\{ f\left( \mathbf{X}\right) -E\left[ f\right] >t\right\} \leq \exp
						\left( \frac{-t^{2}}{2aE\left[ f\left( X\right) \right] +at}\right) \text{.}
						\end{equation*}%
						If in addition $f\left( \mathbf{x}\right) -\inf_{y\in \mathcal{X}%
						}S_{y}^{k}f\left( \mathbf{x}\right) \leq 1$ for all $k\in \left\{
						1,...,N\right\} $ and all $\mathbf{x}\in \mathcal{X}^{N}$ then%
						\begin{equation*}
						\Pr \left\{ E\left[ f\right] -f\left( \mathbf{X}\right) >t\right\} \leq \exp
						\left( \frac{-t^{2}}{2\max \left\{ a,1\right\} E\left[ f\left( X\right) %
							\right] }\right) .
						\end{equation*}
					\end{theorem}
					
					\begin{corollary}
						\label{Corollary selfbound}If $f:\mathcal{X}^{N}\rightarrow \mathbb{R}$
						satisfies (\ref{selfbound condition}) and for some $b>0$ $f\left( \mathbf{x}%
						\right) -\inf_{y\in \mathcal{X}}S_{y}^{k}f\left( \mathbf{x}\right) \leq b$
						for all $k\in \left\{ 1,...,N\right\} $ and all $\mathbf{x}\in \mathcal{X}%
						^{N}$ then 
						\begin{equation*}
						\Pr \left\{ E\left[ f\right] -f\left( \mathbf{X}\right) >t\right\} \leq \exp
						\left( \frac{-t^{2}}{2\max \left\{ a,b\right\} E\left[ f\right] }\right) .
						\end{equation*}%
						Also for all $\delta >0$ with probability at least $1-\delta $%
						\begin{equation*}
						\sqrt{f\left( \mathbf{X}\right) }-\sqrt{2a\ln \left( 2/\delta \right) }\leq 
						\sqrt{E\left[ f\right] }\leq \sqrt{f\left( \mathbf{X}\right) }+\sqrt{2\max
							\left\{ a,b\right\} \ln \left( 2/\delta \right) }.
						\end{equation*}%
						For a one-sided bound $2/\delta $ can be replaced by $1/\delta $.
					\end{corollary}
					
					\begin{proof}
						If $f\left( \mathbf{x}\right) -\inf_{y\in \mathcal{X}}S_{y}^{k}f\left( 
						\mathbf{x}\right) \leq b$ then $\left( f\left( \mathbf{x}\right) /b\right)
						-\inf_{y\in \mathcal{X}}S_{y}^{k}\left( f\left( \mathbf{x}\right) /b\right)
						\leq 1$ and (\ref{selfbound condition}) implies $D^{2}\left( f\left( \mathbf{%
							x}\right) /b\right) \leq \left( a/b\right) \left( f\left( \mathbf{x}\right)
						/b\right) $, so by the second conclusion of Theorem \ref{Theorem selfbound}%
						\begin{eqnarray*}
							\Pr \left\{ E\left[ f\right] -f\left( \mathbf{X}\right) >t\right\}  &=&\Pr
							\left\{ f\left( \mathbf{X}\right) /b-E\left[ f/b\right] >t/b\right\}  \\
							&\leq &\exp \left( \frac{-\left( t/b\right) ^{2}}{2\max \left\{
								a/b,1\right\} E\left[ f/b\right] }\right) =\exp \left( \frac{-t^{2}}{2\max
								\left\{ a,b\right\} E\left[ f\right] }\right) 
						\end{eqnarray*}%
						(this is really an alternative formulation of the second conclusion of
						Theorem \ref{Theorem selfbound}). Equating the R.H.S. to $\delta $ solving
						for $t$ and elementary algebra then give with probability at least $1-\delta 
						$ that%
						\begin{equation*}
						\sqrt{E\left[ f\right] }\leq \sqrt{f\left( \mathbf{X}\right) }+\sqrt{2\max
							\left\{ a,b\right\} \ln \left( 1/\delta \right) }.
						\end{equation*}%
						In a similar way the first conclusion of Theorem \ref{Theorem selfbound}
						gives with probability at least $1-\delta $ that%
						\begin{equation*}
						\sqrt{f\left( \mathbf{X}\right) }-\sqrt{2a\ln \left( 1/\delta \right) }\leq 
						\sqrt{E\left[ f\right] }.
						\end{equation*}%
						A union bound concludes the proof.\bigskip 
					\end{proof}
					
					For the random design we use only the upper tail bound above. The lower tail
					bound will be used for the sub-Gaussian bound Theorem \ref{Theorem SubGauss}.
					
					\begin{lemma}
						\label{Lemma A-B-bounds}For $\delta >0$%
						\begin{eqnarray*}
							\Pr \left\{ \sqrt{A\left( \mathbf{W}\right) }>\sqrt{\frac{m^{2}}{n}}+\left(
							1+4\sqrt{\ln \left( 1/\delta \right) }\right) \sqrt{\frac{m}{M}}\right\} 
							&\leq &\delta  \\
							\Pr \left\{ \sqrt{B\left( \mathbf{W}\right) }>\frac{m^{2}}{n}+\left( 1+4%
							\sqrt{\ln \left( 1/\delta \right) }\right) \frac{m}{\sqrt{M}}\right\}  &\leq
							&\delta 
						\end{eqnarray*}%
						$\bigskip $
					\end{lemma}
					
					\begin{proof}
						To prove these inequalities we show that $A$ and $B$ satisfy a self-bounding
						condition as in (\ref{selfbound condition})$.$
						
						Fix a design $\mathbf{W}$ and for each $j\in \left[ M\right] $ choose $%
						W_{j}^{\prime }$ so as to minimize $A\left( S_{W_{j}^{\prime }}^{j}\mathbf{W}%
						\right) $. Denote $S_{W_{j}^{\prime }}^{j}\mathbf{W}$ by $\mathbf{W}^{\left(
							j\right) }$ in the following. Note that for any $k\in \left[ n\right] $ and $%
						j\in \left[ M\right] $%
						\begin{equation*}
						\sum_{i=1}^{M}Z_{i}^{k}\left( \mathbf{W}\right)
						-\sum_{i=1}^{M}Z_{i}^{k}\left( \mathbf{W}^{\left( j\right) }\right)
						=1\left\{ k\in W_{j}\right\} -1\left\{ k\in W_{j}^{\prime }\right\} ,
						\end{equation*}%
						because the two designs differ only in the $j$-th element. Thus for $j\in %
						\left[ M\right] $%
						\begin{align*}
						& \left( A\left( \mathbf{W}\right) -A\left( \mathbf{W}^{\left( j\right)
						}\right) \right) ^{2} \\
						& =\frac{1}{M^{4}}\left( \sum_{k=1}^{n}\left( \left(
						\sum_{i=1}^{M}Z_{i}^{k}\left( \mathbf{W}\right) \right) ^{2}-\left(
						\sum_{i=1}^{M}Z_{i}^{k}\left( \mathbf{W}^{\left( j\right) }\right) \right)
						^{2}\right) \right) ^{2} \\
						& =\frac{1}{M^{4}}\left( \sum_{k=1}^{n}\left( 1\left\{ k\in W_{j}\right\}
						-1\left\{ k\in W_{j}^{\prime }\right\} \right) \left(
						\sum_{i=1}^{M}Z_{i}^{k}\left( \mathbf{W}\right)
						+\sum_{i=1}^{M}Z_{i}^{k}\left( \mathbf{W}^{\left( j\right) }\right) \right)
						\right) ^{2} \\
						& \leq \frac{1}{M^{4}}\sum_{k=1}^{n}\left( 1\left\{ k\in W_{j}\right\}
						-1\left\{ k\in W_{j}^{\prime }\right\} \right) ^{2}\sum_{k=1}^{n}\left(
						\sum_{i=1}^{M}Z_{i}^{k}\left( \mathbf{W}\right)
						+\sum_{i=1}^{M}Z_{i}^{k}\left( \mathbf{W}^{\left( j\right) }\right) \right)
						^{2} \\
						& \leq \frac{4}{M^{2}}\sum_{k=1}^{n}\left( 1\left\{ k\in W_{j}\right\}
						-1\left\{ k\in W_{j}^{\prime }\right\} \right) ^{2}A\left( \mathbf{W}\right) 
						\\
						& \leq \frac{8m}{M^{2}}A\left( \mathbf{W}\right) .
						\end{align*}%
						The first inequality is Cauchy-Schwarz. The second follows from $\left(
						b+c\right) ^{2}\leq 2b^{2}+2c^{2}$ and the minimality of $\mathbf{W}^{\left(
							j\right) }$ and the last inequality follows from the cardinality constraint
						on the sets $W_{j}$ and $W_{j}^{\prime }$. We conclude that%
						\begin{equation*}
						D^{2}A\left( \mathbf{W}\right) =\sum_{j=1}^{M}\left( A\left( \mathbf{W}%
						\right) -A\left( \mathbf{W}^{\left( j\right) }\right) \right) ^{2}\leq \frac{%
							8m}{M}A\left( \mathbf{W}\right) ,
						\end{equation*}%
						so using Corollary \ref{Corollary selfbound} with $a=8m/M$ we get with
						probability at least $1-\delta $%
						\begin{eqnarray*}
							\sqrt{A\left( \mathbf{W}\right) } &\leq &\sqrt{\mathbb{E}\left[ A\left( 
								\mathbf{W}\right) \right] }+\sqrt{\frac{16m\ln \left( 1/\delta \right) }{M}}
							\\
							&\leq &\sqrt{\frac{m^{2}}{n}}+\left( 1+4\sqrt{\ln \left( 1/\delta \right) }%
							\right) \sqrt{\frac{m}{M}},
						\end{eqnarray*}%
						where we used Lemma \ref{Lemma expectations}.
						
						The proof of the second assertion is similar. Choose $W_{j}^{\prime }$ to
						minimize $B\left( S_{W_{j}^{\prime }}^{j}\mathbf{W}\right) =B\left( \mathbf{W%
						}^{\left( j\right) }\right) $. Then for fixed $j$ we get%
						\begin{eqnarray*}
							&&\left( B\left( \mathbf{W}\right) -B\left( \mathbf{W}^{\left( j\right)
							}\right) \right) ^{2} \\
							&\leq &\frac{1}{M^{4}}\sum_{k\neq l}\left( 1\left\{ k,l\in W_{j}\right\}
							-1\left\{ k,l\in W_{j}^{\prime }\right\} \right) ^{2}\sum_{k\neq l}\left(
							\sum_{i=1}^{M}Z_{i}^{kl}\left( \mathbf{W}\right)
							+\sum_{i=1}^{M}Z_{i}^{kl}\left( \mathbf{W}^{\left( j\right) }\right) \right)
							^{2} \\
							&\leq &\frac{4}{M^{2}}\sum_{k\neq l}\left( 1\left\{ k,l\in W_{j}\right\}
							-1\left\{ k,l\in W_{j}^{\prime }\right\} \right) ^{2}B\left( \mathbf{W}%
							\right)  \\
							&\leq &\frac{8m\left( m-1\right) }{M^{2}}B\left( \mathbf{W}\right) .
						\end{eqnarray*}%
						Analogous to the above we conclude with Lemma \ref{Lemma expectations} that
						with probability at least $1-\delta $%
						\begin{eqnarray*}
							\sqrt{B\left( \mathbf{W}\right) } &\leq &\sqrt{\mathbb{E}\left[ B\left( 
								\mathbf{W}\right) \right] }+\sqrt{\frac{16m\left( m-1\right) \ln \left(
									1/\delta \right) }{M}} \\
							&\leq &\frac{m^{2}}{n}+\left( 1+4\sqrt{\ln \left( 1/\delta \right) }\right) 
							\frac{m}{\sqrt{M}}
						\end{eqnarray*}%
						\bigskip 
					\end{proof}
					
					Assembling the above pieces gives a proof of the bound for random design.
					
					\begin{proof}[Proof of Theorem \protect\ref{Theorem random design}]
						Fix $\delta \in \left( 0,1\right) $. It follows from Lemma \ref{Lemma
							C-bound} that 
						\begin{equation*}
						\Pr \left\{ C>\frac{m}{n}+\frac{\sqrt{2m}+3}{\sqrt{M}}\ln \left( 3/\delta
						_{2}\right) \right\} \leq \frac{\delta _{2}}{3}.
						\end{equation*}%
						From Lemma \ref{Lemma A-B-bounds} we have%
						\begin{eqnarray*}
							\Pr \left\{ \sqrt{A}>\sqrt{\frac{m^{2}}{n}}+\frac{5\sqrt{m}}{\sqrt{M}}\ln
							\left( 3/\delta _{2}\right) \right\}  &\leq &\frac{\delta _{2}}{3}\text{ and}
							\\
							\Pr \left\{ \sqrt{B}>\frac{m^{2}}{n}+\frac{5m}{\sqrt{M}}\ln \left( 3/\delta
							_{2}\right) \right\}  &\leq &\frac{\delta _{2}}{3}.
						\end{eqnarray*}%
						From Theorem \ref{Theorem incomplete U} we get%
						\begin{equation*}
						\Pr \left\{ U_{\mathbf{W}}\left( \mathbf{X}\right) -\theta >\sqrt{A}\sqrt{%
							2\sigma _{1}^{2}\ln \left( 1/\delta _{1}\right) }+\left( \alpha \sqrt{B}%
						+\left( 4/3\right) C\right) \ln \left( 1/\delta _{1}\right) \right\} \leq 
						\frac{\delta _{1}}{4}.
						\end{equation*}%
						Combining the last four inequalities in a union bound, using $\sigma
						_{1}^{2}\leq 1$ and simplifying, gives with probability at least $1-\delta $
						that%
						\begin{equation*}
						U_{\mathbf{W}}\left( \mathbf{X}\right) -\theta \leq \sqrt{\frac{2m^{2}\sigma
								_{1}^{2}\ln \left( 1/\delta _{1}\right) }{n}}+\frac{\alpha m^{2}+\left(
							4/3\right) m}{n}\ln \left( 1/\delta _{1}\right) +\frac{5\alpha m+9\sqrt{m}+4%
						}{\sqrt{M}}\ln ^{2}\left( 3/\delta _{2}\right) .
						\end{equation*}%
						\bigskip 
					\end{proof}
					
					\subsection{Proof of the sub-Gaussian lower bound\protect\bigskip }
					
					\begin{proof}[Proof of Theorem \protect\ref{Theorem SubGauss}]
						The proof is based on the lower-tail bound of Corollary \ref{Corollary
							selfbound}. For each $k$ choose $x_{k}$ so as to minimize $U_{\mathbf{W}%
						}\left( \mathbf{x}_{\left( k\right) }\right) $. This means that%
						\begin{equation*}
						0\leq \sum_{i=1}^{M}K\left( \mathbf{x}^{W_{i}}\right) -K\left( \mathbf{x}%
						_{\left( k\right) }^{W_{i}}\right) =\sum_{i:k\in W_{i}}K\left( \mathbf{x}%
						^{W_{i}}\right) -K\left( \mathbf{x}_{\left( k\right) }^{W_{i}}\right) .
						\end{equation*}%
						Consequently%
						\begin{equation}
						\sum_{i:k\in W_{i}}K\left( \mathbf{x}_{\left( k\right) }^{W_{i}}\right) \leq
						\sum_{i:k\in W_{i}}K\left( \mathbf{x}^{W_{i}}\right) .  \label{bbb}
						\end{equation}%
						Now%
						\begin{align*}
						& \sum_{k=1}^{n}\left( U_{\mathbf{W}}\left( \mathbf{x}\right) -U_{\mathbf{W}%
						}\left( \mathbf{x}_{\left( k\right) }\right) \right) ^{2} \\
						& =\frac{1}{M^{2}}\sum_{k=1}^{n}\left( \sum_{i:k\in W_{i}}\left( K\left( 
						\mathbf{x}^{W_{i}}\right) -K\left( \mathbf{x}_{\left( k\right)
						}^{W_{i}}\right) \right) \right) ^{2} \\
						& =\frac{1}{M^{2}}\sum_{k=1}^{n}\left( \sum_{i:k\in W_{i}}\left( \sqrt{K}%
						\left( \mathbf{x}^{W_{i}}\right) -\sqrt{K}\left( \mathbf{x}_{\left( k\right)
						}^{W_{i}}\right) \right) \left( \sqrt{K}\left( \mathbf{x}^{W_{i}}\right) +%
						\sqrt{K}\left( \mathbf{x}_{\left( k\right) }^{W_{i}}\right) \right) \right)
						^{2} \\
						& \leq \frac{1}{M^{2}}\sum_{k=1}^{n}\sum_{i:k\in W_{i}}\left( \sqrt{K}\left( 
						\mathbf{x}^{W_{i}}\right) -\sqrt{K}\left( \mathbf{x}_{\left( k\right)
						}^{W_{i}}\right) \right) ^{2}\sum_{i:k\in W_{i}}\left( \sqrt{K}\left( 
						\mathbf{x}^{W_{i}}\right) +\sqrt{K}\left( \mathbf{x}_{\left( k\right)
						}^{W_{i}}\right) \right) ^{2} \\
						& \leq \frac{1}{M}\max_{k}\left\vert \sum_{i:k\in W_{i}}\left( \sqrt{K}%
						\left( \mathbf{x}^{W_{i}}\right) -\sqrt{K}\left( \mathbf{x}_{\left( k\right)
						}^{W_{i}}\right) \right) ^{2}\right\vert \frac{1}{M}\sum_{k=1}^{n}\sum_{i:k%
						\in W_{i}}\left( \sqrt{K}\left( \mathbf{x}^{W_{i}}\right) +\sqrt{K}\left( 
					\mathbf{x}_{\left( k\right) }^{W_{i}}\right) \right) ^{2}
					\end{align*}%
					We bound the second factor, using $\left( a-b\right) ^{2}\leq 2a^{2}+2b^{2}$
					and (\ref{bbb})%
					\begin{align*}
					& \frac{1}{M}\sum_{k=1}^{n}\sum_{i:k\in W_{i}}\left( \sqrt{K}\left( \mathbf{x%
					}^{W_{i}}\right) +\sqrt{K}\left( \mathbf{x}_{\left( k\right)
				}^{W_{i}}\right) \right) ^{2} \\
				& \leq \frac{2}{M}\sum_{k=1}^{n}\left( \sum_{i:k\in W_{i}}K\left( \mathbf{x}%
				^{W_{i}}\right) +\sum_{i:k\in W_{i}}K\left( \mathbf{x}_{\left( k\right)
				}^{W_{i}}\right) \right)  \\
				& \leq \frac{4}{M}\sum_{k=1}^{n}\sum_{i:k\in W_{i}}K\left( \mathbf{x}%
				^{W_{i}}\right) =\frac{4}{M}\sum_{i}\sum_{k\in W_{i}}K\left( \mathbf{x}%
				^{W_{i}}\right)  \\
				& =4mU_{\mathbf{W}}\left( \mathbf{x}\right) .
				\end{align*}%
				Thus if $\sqrt{K}\in \left[ 0,1\right] $%
				\begin{eqnarray*}
					\sum_{k=1}^{n}\left( U_{\mathbf{W}}\left( \mathbf{x}\right) -U_{\mathbf{W}%
					}\left( \mathbf{x}_{\left( k\right) }\right) \right) ^{2} &\leq &\frac{4m}{M}%
					\max_{k}\left\vert \sum_{i:k\in W_{i}}\left( \sqrt{K}\left( \mathbf{x}%
					^{W_{i}}\right) -\sqrt{K}\left( \mathbf{x}_{\left( k\right) }^{W_{i}}\right)
					\right) ^{2}\right\vert U_{\mathbf{W}}\left( \mathbf{x}\right)  \\
					&\leq &\frac{4m}{M}\max_{k}R_{k}U_{\mathbf{W}}\left( \mathbf{x}\right)
					=4mC\left( \mathbf{W}\right) U_{\mathbf{W}}\left( \mathbf{x}\right) .
				\end{eqnarray*}%
				So we can use $a=4mC\left( \mathbf{W}\right) $ in Corollary \ref{Corollary
					selfbound}. We also have%
				\begin{equation*}
				\frac{1}{M}\sum_{i=1}^{M}K\left( \mathbf{x}^{W_{i}}\right) -K\left( \mathbf{x%
				}_{\left( k\right) }^{W_{i}}\right) \leq \frac{1}{M}\sum_{i:k\in
				W_{i}}K\left( \mathbf{x}^{W_{i}}\right) -K\left( \mathbf{x}_{\left( k\right)
			}^{W_{i}}\right) \leq C\left( \mathbf{W}\right) ,
			\end{equation*}%
			So with $b=C\left( \mathbf{W}\right) $ and $\max \left\{ a,b\right\}
			=4mC\left( \mathbf{W}\right) $ Corollary \ref{Corollary selfbound} gives%
			\begin{equation*}
			\Pr \left\{ \mathbb{E}\left[ U_{\mathbf{W}}\right] -U_{\mathbf{W}}\left( 
			\mathbf{X}\right) >t\right\} \leq \exp \left( \frac{-t^{2}}{8mC~\mathbb{E}%
				\left[ U_{\mathbf{W}}\right] }\right) .
			\end{equation*}
		\end{proof}
		
		\begin{lemma}
			\label{Lemma bound Beta and Gamma}If $K:\mathbf{x}\in \mathcal{X}%
			^{m}\rightarrow \left[ -1,1\right] $ then $\beta \left( K\right) \leq \gamma
			\left( K\right) \leq 8.$
		\end{lemma}
		
		\begin{proof}
			The first inequality is obvious. We have 
			\begin{eqnarray*}
				\beta \left( K\right)  &\leq &\gamma \left( K\right) =\mathbb{E}\left[
				\left( D_{y,y^{\prime }}^{1}D_{X_{2},X_{2}^{\prime }}^{2}K\left(
				x_{1},x_{2},x_{3},...,x_{m}\right) \right) ^{2}\right]  \\
				&\leq &4~\sup_{\mathbf{x}\in \mathcal{X}^{m}}\mathbb{E}\left[ \left(
				D_{X_{1},X_{1}^{\prime }}^{2}K\left( x_{1},x_{2},x_{3},...,x_{m}\right)
				\right) ^{2}\right]  \\
				&\leq &8~\sup_{\mathbf{x}\in \mathcal{X}^{m}}\text{Var}\left[ K\left(
				X_{1},...,X_{m}\right) |X_{2}=x_{2},...,X_{m}=x_{m}\right]  \\
				&\leq &8.
			\end{eqnarray*}
		\end{proof}

\newpage

\bibliographystyle{plain}

\begin{thebibliography}{10}
	
	\bibitem{adamczak2006moment}
	Rados{\l}aw Adamczak et~al.
	\newblock Moment inequalities for u-statistics.
	\newblock {\em The Annals of Probability}, 34(6):2288--2314, 2006.
	
	\bibitem{arcones1995bernstein}
	Miguel~A Arcones.
	\newblock A bernstein-type inequality for u-statistics and u-processes.
	\newblock {\em Statistics \& probability letters}, 22(3):239--247, 1995.
	
	\bibitem{arcones1993limit}
	Miguel~A Arcones and Evarist Gin{\'e}.
	\newblock Limit theorems for u-processes.
	\newblock {\em The Annals of Probability}, pages 1494--1542, 1993.
	
	\bibitem{blom1976some}
	Gunnar Blom.
	\newblock Some properties of incomplete u-statistics.
	\newblock {\em Biometrika}, 63(3):573--580, 1976.
	
	\bibitem{chen2019randomized}
	Xiaohui Chen and Kengo Kato.
	\newblock Randomized incomplete $ u $-statistics in high dimensions.
	\newblock {\em The Annals of Statistics}, 47(6):3127--3156, 2019.
	
	\bibitem{clemenccon2016scaling}
	St{\'e}phan Cl{\'e}men{\c{c}}on, Igor Colin, and Aur{\'e}lien Bellet.
	\newblock Scaling-up empirical risk minimization: Optimization of incomplete
	u-statistics.
	\newblock {\em The Journal of Machine Learning Research}, 17(1):2682--2717,
	2016.
	
	\bibitem{clemenccon2013maximal}
	St{\'e}phan Cl{\'e}men{\c{c}}on, Sylvain Robbiano, and Jessica Tressou.
	\newblock Maximal deviations of incomplete u-statistics with applications to
	empirical risk sampling.
	\newblock In {\em Proceedings of the 2013 SIAM International Conference on Data
		Mining}, pages 19--27. SIAM, 2013.
	
	\bibitem{Efron81}
	B.~Efron and C.~Stein.
	\newblock The jackknife estimate of variance.
	\newblock {\em The Annals of Statistics}, pages 586--596, 1981.
	
	\bibitem{gine2000exponential}
	Evarist Gin{\'e}, Rafa{\l} Lata{\l}a, and Joel Zinn.
	\newblock Exponential and moment inequalities for u-statistics.
	\newblock In {\em High Dimensional Probability II}, pages 13--38. Springer,
	2000.
	
	\bibitem{halmos1946theory}
	Paul~R Halmos.
	\newblock The theory of unbiased estimation.
	\newblock {\em The Annals of Mathematical Statistics}, 17(1):34--43, 1946.
	
	\bibitem{hoeffding58probability}
	W~Hoeffding.
	\newblock Probability inequalities for sums of bounded random variables.
	\newblock {\em Journal of the American Statistical Association}, 58:301, 1963.
	
	\bibitem{hoeffding1948class}
	Wassily Hoeffding.
	\newblock A class of statistics with asymptotically normal distribution.
	\newblock {\em The annals of mathematical statistics}, pages 293--325, 1948.
	
	\bibitem{Houdre97}
	C.~Houdr{\'e}.
	\newblock The iterated jackknife estimate of variance.
	\newblock {\em Statistics and probability letters}, 35(2):197--201, 1997.
	
	\bibitem{kong2020design}
	Xiangshun Kong and Wei Zheng.
	\newblock Design based incomplete u-statistics.
	\newblock {\em arXiv preprint arXiv:2008.04348}, 2020.
	
	\bibitem{lee1982incomplete}
	Alan~J Lee.
	\newblock On incomplete u-statistics having minimum variance.
	\newblock {\em Australian Journal of Statistics}, 24(3):275--282, 1982.
	
	\bibitem{maurer2006concentration}
	Andreas Maurer.
	\newblock Concentration inequalities for functions of independent variables.
	\newblock {\em Random Structures \& Algorithms}, 29(2):121--138, 2006.
	
	\bibitem{maurer2019bernstein}
	Andreas Maurer.
	\newblock A bernstein-type inequality for functions of bounded interaction.
	\newblock {\em Bernoulli}, 25(2):1451--1471, 2019.
	
	\bibitem{McDiarmid98}
	C.~McDiarmid.
	\newblock Concentration.
	\newblock In {\em Probabilistic Methods of Algorithmic Discrete Mathematics},
	pages 195--248, Berlin, 1998. Springer.
	
\end{thebibliography}

\end{document}